\documentclass[12 pt]{amsart}
\usepackage{amscd,amsmath,amsthm,amssymb}
\usepackage{verbatim}
\usepackage[all]{xy}
\usepackage{color}
\usepackage{hyperref}

\usepackage{tikz, float} \usetikzlibrary {positioning}

\usepackage[lined,algonl,boxed,norelsize]{algorithm2e}
\let\chapter\undefined

\def\frk{\mathfrak}               

\def\Phi{{\frk n}}
\def\Phi{{\frk N}}

\def\HS{{\mathcal HS}}

\def\xb{{\mathbf x}}
\def\yb{{\mathbf y}}

\def\A{{\mathcal A}}

\def\H{{\mathcal H}}

\def\P{{\mathcal P}}

\def\C{{\mathcal C}}

\def\xb{{\mathbf x}}
\def\yb{{\mathbf y}}

\def\qed{\ifhmode\textqed\fi
      \ifmmode\ifinner\quad\qedsymbol\else\dispqed\fi\fi}
\def\textqed{\unskip\nobreak\penalty50
       \hskip2em\hbox{}\nobreak\hfil\qedsymbol
       \parfillskip=0pt \finalhyphendemerits=0}
\def\dispqed{\rlap{\qquad\qedsymbol}}

\def\opn#1#2{\def#1{\operatorname{#2}}} 
\opn\chara{char} \opn\length{\ell} \opn\pd{pd} \opn\rk{rk}
\opn\projdim{proj\,dim} \opn\injdim{inj\,dim} \opn\rank{rank}
\opn\depth{depth} \opn\grade{grade} \opn\height{height}
\opn\embdim{emb\,dim} \opn\codim{codim}

\opn\Tr{Tr} \opn\bigrank{big\,rank}
\opn\superheight{superheight}\opn\lcm{lcm}
\opn\trdeg{tr\,deg}
\opn\reg{reg} \opn\lreg{lreg} \opn\ini{in} \opn\lpd{lpd}
\opn\size{size} \opn\sdepth{sdepth}
\opn\link{link}\opn\fdepth{fdepth}\opn\lex{lex}
\opn\LM{LM}
\opn\LC{LC}
\opn\NF{NF}
\opn\Merge{Merge}
\opn\sgn{sgn}
\opn\div{div} \opn\Div{Div} \opn\cl{cl} \opn\Pic{Pic}
\opn\Prin{Prin}
\opn\op{op}
\opn\indeg{indeg} \opn\outdeg{outdeg}
\opn\red{red}
\opn\Spec{Spec} \opn\Supp{Supp} \opn\supp{supp} \opn\Sing{Sing}
\opn\Ass{Ass} \opn\Min{Min}\opn\Mon{Mon}
\opn\Ann{Ann} \opn\Rad{Rad} \opn\Soc{Soc}
 \opn\Ker{Ker} \opn\Coker{Coker} \opn\Am{Am}
\opn\Hom{Hom} \opn\Tor{Tor} \opn\Ext{Ext} \opn\End{End}
\opn\Aut{Aut} \opn\id{id}

\opn\nat{nat}
\opn\pff{pf}
\opn\Pf{Pf} \opn\GL{GL} \opn\SL{SL} \opn\mod{mod} \opn\ord{ord}
\opn\Gin{Gin} \opn\Hilb{Hilb}\opn\sort{sort}
\opn\span{span}
\opn\Image{Image}
\opn\aff{aff} \opn\con{conv} \opn\relint{relint} \opn\st{st}
\opn\lk{lk} \opn\cn{cn} \opn\core{core} \opn\vol{vol}
\opn\link{link} \opn\star{star}\opn\lex{lex}\opn\set{set}
\opn\dist{dist}
\opn\gr{gr}

\def\pot#1#2{#1[\kern-0.28ex[#2]\kern-0.28ex]}

\opn\dirlim{\underrightarrow{\lim}}
\opn\inivlim{\underleftarrow{\lim}}
\let\to=\rightarrow

\def\Implies{\ifmmode\Longrightarrow \else
        \unskip${}\Longrightarrow{}$\ignorespaces\fi}
\def\implies{\ifmmode\Rightarrow \else
        \unskip${}\Rightarrow{}$\ignorespaces\fi}
\def\iff{\ifmmode\Longleftrightarrow \else
        \unskip${}\Longleftrightarrow{}$\ignorespaces\fi}

\let\:=\colon
\newtheorem{Theorem}{Theorem}[section]
\newtheorem{Lemma}[Theorem]{Lemma}
\newtheorem{Corollary}[Theorem]{Corollary}
\newtheorem{Proposition}[Theorem]{Proposition}

\theoremstyle{remark}
\newtheorem{Remark}[Theorem]{Remark}

\theoremstyle{definition}
\newtheorem{Example}[Theorem]{Example}

\newtheorem{Definition}[Theorem]{Definition}
\newtheorem*{Notation}{Notation}
\let\kappa=\varkappa
\def\qed{\ifhmode\textqed\fi
      \ifmmode\ifinner\quad\qedsymbol\else\dispqed\fi\fi}
\def\textqed{\unskip\nobreak\penalty50
       \hskip2em\hbox{}\nobreak\hfil\qedsymbol
       \parfillskip=0pt \finalhyphendemerits=0}
\def\dispqed{\rlap{\qquad\qedsymbol}}

\opn\dis{dis}
\def\pnt{{\raise0.5mm\hbox{\large\bf.}}}

\opn\Lex{Lex}
\opn\syz{{\rm syz}}
\opn\spoly{{\rm spoly}}
\opn\LM{{\rm LM}}
\opn\lm{{\rm lm}}
\opn\lcm{{\rm lcm}} \opn\A{\mathcal A}

\numberwithin{equation}{section}

\tikzstyle{Cwhite}=[scale = .6,circle, fill = white, minimum size=2.5mm]
\tikzstyle{Cgray}=[scale = .4,circle, fill = gray, minimum size=3mm]
\tikzstyle{Cblack2}=[scale = .4,circle, fill = black, minimum size=3mm]
\tikzstyle{Cblack}=[scale = .7,circle, fill = black, minimum size=3mm]
\tikzstyle{C0}=[scale = .9,circle, fill = black!0, inner sep = 0pt, minimum size=3mm]
\tikzstyle{C1}=[scale = .7,circle, fill = black!0, inner sep = 0pt, minimum size=3mm]
\tikzstyle{Cred}=[scale = .4,circle, fill = red, minimum size=3mm]
\tikzstyle{Cblue}=[scale = .4,circle, fill =blue, minimum size=3mm]

\begin{document}

\title{The algebraic method in tree percolation}
\author{Fatemeh Mohammadi} 
\address{Institut f\"ur Mathematik\\ Technische Universit\"at Berlin\\ 10623 Berlin, Germany}
\email{mohammad@math.tu-berlin.de}

\author{Eduardo S\'aenz-de-Cabez\'on}
\address{Department of Mathematics,  Universidad de La Rioja, Spain}
\email{eduardo.saenz-de-cabezon@unirioja.es}

\author{Henry P. Wynn}
\address{Department of Statistics, London School of Economics, UK}
\email{h.wynn@lse.ac.uk}


\maketitle
\begin{abstract}
We apply the methods of algebraic reliability to the study of percolation on trees. To a complete $k$-ary tree $T_{k,n}$ of depth $n$ we assign a monomial ideal $I_{k,n}$ on $\sum_{i=1}^n k^i$  variables and $k^n$ minimal monomial generators. We give explicit recursive formulae for the Betti numbers of $I_{k,n}$ and their Hilbert series, which allow us to study explicitly percolation on $T_{k,n}$. We study bounds on this percolation and study its asymptotical behavior with the mentioned commutative algebra techniques.
\end{abstract}
\section{Introduction}\label{sec:intro}

The study of monomial ideals has experienced a big development in the last couple of decades, not only from a theoretical point of view \cite{HerzogHibi} but also from the point of view of applications and algorithms \cite{BGS13}. Of particular interest are the relations between the algebra of monomial ideals and the combinatorics of graphs and networks \cite{Villarreal, morey&Villarreal, V13}. In relation with these lines of research, the authors have developed an algebraic theory of system reliability which can be applied to industrial, biological and communication systems, among others \cite{GW04,SW09,SW14,SW14b}. In this theory, a monomial ideal is associated to a coherent system and the study of the reliability of the system is performed by studying algebraic invariants of the ideal, such as the Hilbert series and Betti numbers. This algebraic approach to system reliability analysis is an example of enumerative methods for reliability evaluation. In particular, it is an improvement of the inclusion-exclusion method, which is the most general one for coherent systems \cite{GW04,SW09}.

A main difficulty and the first step  in the use of monomial ideals to study the reliability of coherent systems is the enumeration of the working and failure states of the system. This made the authors focus on several widely used and structured systems, like $k$-out-of-$n$ systems \cite{SW09}, series-parallel systems \cite{Sw11}, all-terminal networks \cite{FatemehFarbod,FatemehFarbod2,Fatemeh2014}, and the more general category of two-terminal networks \cite{SW10,Fatemeh2014}. The present paper follows this line extending the application of the algebraic approach to reliability analysis to a more general situation, which allows us to  introduce these techniques in  percolation theory, a branch of probability theory.

In the setting of two-terminal networks the situation is the following. Consider a network as a simple connected graph $G=(V,E)$, where $V$ is the set of vertices (nodes) and $E$ is the set of edges (connections). To have a two-terminal network, we select two special vertices in the graph, $s$ (source) and $t$ (target) and study the connections between $s$ and $t$ in the network. We consider that vertices are reliable but edges may fail. The network fails to communicate between $s$ and $t$ whenever there is a set of failing (removed) edges such that there is no path connecting $s$ and $t$ using only the remaining edges. Such a set of edges is called a {\em cut} in this context. On the other hand, a {\em path} is a set of working edges that connect $s$ and $t$. We say that the network is working whenever there is a path of working edges between $s$ and $t$. In the algebraic approach we consider a polynomial ring on $n$ variables, where $n$ is the number of edges of $G$, i.e. $n=\vert E\vert$. We associate a variable $x_e$, $e\in\{1,\dots,n\}$ to each edge $x_e$ in $E$. To a set of edges we associate the product of their corresponding variables. The main observation in the algebraic approach to network reliability is that the monomials corresponding to the set of cuts (respectively paths) of a network $G$ generate a monomial ideal, which we call the cut ideal of $G$, $J_G$ (respectively the path ideal of $G$, $I_G$). The evaluation of the (numerator of the) Hilbert series of either the cut ideal or the path ideal of $G$ using the probabilities of failure or function of each edge and their combinations, gives us the reliability of $G$. Furthermore, if we consider the form of the Hilbert series given by a free resolution of the ideal, we can obtain bounds for the reliability of $G$, which are tighter than the usual Bonferroni bounds \cite{SW09}.

The outline of the paper is the following. In \S\ref{sec:preliminary} we  generalize this setting to any situation in which a cut and a path are defined in opposition to each other, in an obvious way: a cut separates  a designated set of pairs of vertices and a path connects all such pairs. This allows us to study the problem of {\em all-terminal reliability} and {\em multi-source multi-terminal reliability}. These more general situations include the setting of {\em percolation theory}. In \S\ref{sec:tree} we apply this method to study percolation on complete trees. This is a new and relevant application of the algebraic method in reliability. We describe the path and cut ideals in this case and compute exact Hilbert series and Betti numbers.  We also give and compute recursive formulas for them. With these results at hand, we study in \S\ref{sec:bounds} path and cut bounds for percolation in trees, recover some classical results on critical values and study the asymptotic behaviour of percolation on trees and their corresponding Betti numbers.

\medskip

{\bf Acknowledgements.} The first author would like to thank J\"urgen Herzog for the helpful discussion about Lemma~\ref{initial2}. We would like to thank the referees for carefully reading our manuscript and their very helpful suggestions and remarks. The second and third authors were partially supported by Ministerio de Econom\'ia y Competitividad, Spain, under grant MTM2013-41775-P.

\section{Monomial Ideals, Betti  numbers and tight inclusion-exclusion bounds}\label{sec:preliminary}
\begin{Definition}\label{def:cut-path-ideals}
Given two disjoint nonempty subsets $A,B$ of $V(G)$ we define
\[
E(A,B) = \{e\in E(G): e\cap A\neq\emptyset\ \textit{and}\ \ e \cap B\neq\emptyset\}.
\]
For a nonempty subset $A$ of $V(G)$, $E(A,A^c)$ is called a {\em cut} of $G$.
A cut $E(A,A^c)$ is called {\em connected} if $G[A]$ and $G[A^c]$ are connected, where $G[A]$ denotes the induced subgraph of $G$ with the vertex set $A$. A cut which is minimal with respect to inclusion is called minimal.
\end{Definition}
Fix a vertex $q$ of $G$ as a source, and fix a subset $L\subseteq V(G)\backslash\{q\}$ as targets. Let $\mathcal{S}_{L,q}$ be the set containing all connected cuts $E(A, A^c)$ of $G$, with  $L\subset  A$ and $q\in  A^c$, and let $\mathcal{D}_{L,q}$ be the set containing all paths between $q$ and one of the vertices of $L$.

\medskip

Let $K$ be a field and let $S=K[\xb]$ be the polynomial ring in the $n=|E(G)|$ variables $\{x_e: e \in E(G)\}$.
We associate the monomial $m_C=\prod_{e\in C} x_e$ to each cut $C=E(A,A^c)$, and  the monomial $m_P=\prod_{e\in P} x_e$ to each path $P$.
We will be concerned with the following ideals in $R$:
\[
\C_{L,q}=\langle m_C:\  C\in \mathcal{S}_{L,q} \rangle\quad\text{and}\quad
\P_{L,q}=\langle m_P:\  P\in \mathcal{D}_{L,q} \rangle.
\]

\begin{Example}{\rm
Consider the two-terminal network $G$ depicted in Figure~\ref{fig:doublebridge}, known as the {\em double bridge} network. We have $E_G=\{ 12,13,14,23,25,34,35,45\}$. Consider vertex $q=1$ as the source and let $L=\{5\}$ be the set of targets.  Following the notation in \ref{def:cut-path-ideals} we obtain the following table of cuts
}
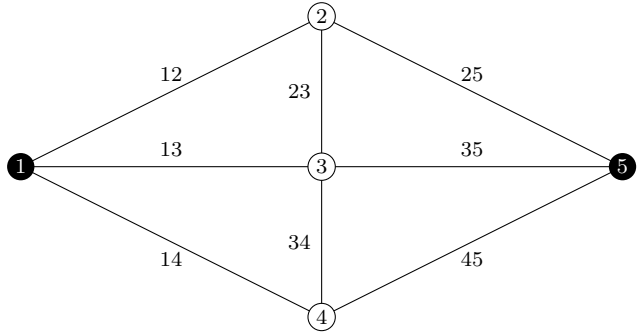
\begin{figure}[h]
\centering
\begin{tikzpicture}
\draw (0,0) -- (4,2) node[pos=0.5, above]{\tiny $12$};
\draw (0,0) -- (4,0) node[pos=0.5, above]{\tiny $13$};
\draw (0,0) -- (4,-2) node[pos=0.5, below]{\tiny $14$};
\draw (4,2) -- (4,0) node[pos=0.5, left]{\tiny $23$};
\draw (4,0) -- (4,-2) node[pos=0.5, left]{\tiny $34$};
\draw (4,2) -- (8,0) node[pos=0.5, above]{\tiny $25$};
\draw (4,0) -- (8,0) node[pos=0.5, above]{\tiny $35$};
\draw (4,-2) -- (8,0) node[pos=0.5, below]{\tiny $45$};
\fill[fill=black] (0,0) circle (1ex) node[text=white]{\tiny{1}};
\fill[fill=white,draw=black] (4,2) circle (1ex) node{\tiny{2}};
\fill[fill=white,draw=black](4,0) circle (1ex) node {\tiny{3}};
\fill[fill=white,draw=black] (4,-2) circle (1ex) node {\tiny{4}};
\fill[fill=black] (8,0) circle (1ex) node[text=white]{\tiny{5}};
\end{tikzpicture}
\caption{Double bridge network.}\label{fig:doublebridge}

\end{figure}
\vspace{-2mm}
\begin{center}
\begin{tabular}{llll}
$A$&$A^c$&$E(A,A^c)$\\
\hline
\{$5$\}&\{$1,2,3,4$\}&\{$24,35,45$\}\\
\{$2,5$\}&\{$1,3,4$\}&\{$12,23,35,45$\}\\
\{$3,5$\}&\{$1,2,4$\}&\{$13,23,34,25,45$\}\\
\{$4,5$\}&\{$1,2,3$\}&\{$14,25,34,35$\}\\
\{$2,3,5$\}&\{$1,4$\}&\{$12,13,34,45$\}\\
\{$2,4,5$\}&\{$1,3$\}&\{$12,23,14,34,35$\}\\
\{$3,4,5$\}&\{$1,2$\}&\{$13,23,14,25$\}\\
\{$2,3,4,5$\}&\{$1$\}&\{$12,13,14$\}\\
\end{tabular}
\end{center}
{\rm Hence } $\C_{L,q}=\langle x_{25}x_{35}x_{45},x_{12}x_{23}x_{35}x_{45},x_{13}x_{23}x_{25}x_{34}x_{45},x_{14}x_{25}x_{34}x_{35},x_{12}x_{13}x_{34}x_{45},$\\ $x_{12}x_{14}x_{23}x_{34}x_{35},x_{13}x_{14}x_{23}x_{25},x_{12}x_{13}x_{14}\rangle $ {\rm which is exactly the cut ideal of $G$ in the two-terminal setting.
}
\end{Example}
\vspace{-2mm}
\begin{Remark}\label{rem:dual}
{\rm
Alexander duality will be very useful in this context  (see \cite[Def.~5.20]{MillerSturmfels}). We recall that the squarefree Alexander dual of
$I=\langle \xb^{a_1},\ldots, \xb^{a_r} \rangle$ is the ideal $I^*= \underline{m}^{a_1}\cap\cdots\cap  \underline{m}^{a_r}$, where $\xb^{a_i}=\prod_{a_i^j\neq 0} x_j$ and  $\underline{m}^{a}=\langle x_j:a_i^j\neq 0\rangle$ for each vector $a_i=(a_i^1,\dots,a_i^n)\in \mathbb{N}^n$ which is a zero-one vector.
}
\end{Remark}
In this setting, we show that the path ideal is the Alexander dual of the cut ideal.
Let us give a brief reminder adapted to our setting. Let $\Sigma_G$ denote the associated simplicial complex to $\P_{L,q}$ on the vertices $\{x_e: e \in E(G)\}$.  The following result is a slight generalization of \cite[Prop. 8.1]{FatemehFarbod2}.
\begin{Proposition}\label{prop:primdecom}
The number of facets of $\Sigma_G$ is the same as the number of minimal cuts  of $G$.
For each cut $C$, the corresponding facet $\tau_C$ is
$
\tau_C = \{x_e: e \in E(G) \backslash C\}.
$
The minimal prime decomposition of $\P_{L,q}$ is
\[
\P_{L,q} = \bigcap_{C\in {\mathcal S}_{L,q}} \langle x_e: e \in C \rangle,
\]
the intersection being over all minimal cuts of $G$. In particular, $\P_{L,q} $ is the Alexander dual of  $\C_{L,q}$.
\end{Proposition}

\begin{proof}
The ideal $\P_{L,q}$ is generated by monomials $\prod_{e \in P}{x_e}$, where $P$ is a path from $q$ to one of the vertices in $L$.
First we show that for each cut $C$, the monomial $m_{\bar{C}}:=\prod_{e \in E(G)\backslash C}{x_e}$ does not belong to $\P_{L,q} $.
Clearly $m_{\bar{C}} \in \P_{L,q} $
if and only if $m_{\bar{C}}$ is divisible by one of the given generators $\prod_{e \in P}{x_e}$. But
\[
\prod_{e \in P}{x_e} \,\,\, \mid \, \prod_{e \in E(G)\backslash C}{x_e} \quad\quad \iff \quad\quad  P \subseteq (E(G)\backslash C) \ .
\]
However, it follows from the definition of cuts that $E(G)\backslash C$ does not contain any path from $q$ to any element of $L$.
This shows that $\tau_C =\{x_e: e \in E(G) \backslash C\}$ is a face in the simplicial complex $\Sigma_G$.
Next we show that $\tau_C$ must be a facet; for $f \in C$, because $C$ is a minimal cut of $G$, $G[C\backslash\{f\}]$ still has a path between $q$ and some element of $L$. 
Then the monomial $m_C \cdot x_f$ is divisible by $\prod_{e \in P}{x_e}$.

It remains to show that for any monomial $m = \prod_{e \in F}{x_e}$ that does not belong to $\P_{L,q}$ we have $F\subseteq (E(G) \backslash C)$ for some cut $C$.
To show this, we repeatedly use the fact that $m$ is not divisible by generators of the form $\prod_{e \in P}{x_e}$ for various $P$, and construct a cut $C$.
Note that if $\prod_{x \in F}{x_e}$ is not divisible by $\prod_{e\in P}{x_e}$ then there exists an $e \in P$ such that $e \not \in  F$. We consider the set consisting all such edges which is clearly a (not necessarily minimal) cut.
The proof now is complete by  \cite[Thm.~1.7]{MillerSturmfels}.
\end{proof}
\medskip


Let us now step into probability theory. In order to apply monomial algebra to network reliability, we assign a working probability to each of the connections (edges) of our network (graph). We shall consider that each edge $e$ operates with independent probability $p_e$ and fails with probability $q_e=1-p_e$. Our task is then to compute the probability $\mathcal{P}(p_e)$ that the system operates (at least one path) or fails (at least one cut) which is $1-\mathcal{P}(p_e)$. For these computations, we use the numerator of the Hilbert series of the path or cut ideals. Note that we could also consider 
dependent probabilities for each edge. This would need more complicated computations but not different methods.

The multigraded Hilbert series of $S/I$ for an ideal $I$ can be expressed in terms of the multidegrees of the modules in any multigraded resolution of $S/I$, as
$$\HS_I(x,t)=\frac{1+\sum_{i=1}^d (-1)^ix^i(\sum_{\alpha \in {
\mathbb{N}^n}} \gamma_{i,\alpha} t^{\alpha})}{\prod_{j=1}^n(1-t_i)},
$$
where the $\gamma_{i,\alpha}$ are the ranks of the multigraded modules in the resolution. If the resolution is minimal then
\begin{eqnarray}\label{hilb}
{\HS_I(x,t)}=\frac{1+\sum_{i=1}^d (-1)^i x^i(\sum_{\alpha \in {
\mathbb{N}^n}} \beta_{i,\alpha} t^{\alpha})}{\prod_{j=1}^n(1-t_i)},
\end{eqnarray}
where the
$\beta_{i,\alpha}$ depend only on $I$ and are known as the
\emph{multigraded Betti numbers} of $S/I$. Observe that the minimality of the resolution means that
\begin{equation*}
\beta_{i,\alpha} \leq \gamma_{i,\alpha}\quad \forall \alpha,i.
\end{equation*}
To simplify our notation, we set
\begin{eqnarray}
 \H_I(x,t)=-\sum_{i=1}^d (-1)^i x^i(\sum_{\alpha \in {\mathbb{N}^n}} \beta_{i,\alpha} t^{\alpha}), \
\end{eqnarray}
and we refer to this as the numerator of the Hilbert series of $I$ which can be seen as a special kind of inclusion-exclusion formulae for counting the monomials in the union of the
 ideals based on each individual minimal generator. By truncating at different homological degrees, or ``depths" $i$ we obtain
successive upper and lower bounds for the indicator function of this union. A key point here is that the bounds given  by the minimal resolution
are tighter, or at least as tight as, those given by the classical Bonferroni  (inclusion-exclusion) bounds. In the algebraic setting the Bonferroni bounds  correspond to the computation of the Hilbert series of $S/I$ using the Taylor resolution, see \cite{SW09} for a full explanation.

Observe that if we have that each edge $e$ of $G$ has a different operating probability $p_e$ then we need the multigraded version of the Hilbert series of $S/I$ to obtain the bounds and probability formulae for the reliability of the network. However, if all the edges operate (do not fail) independently with the same probability $p$, i.e., if $p_e=p$ for each edge $e$ of $G$, then we only need  the graded Betti numbers, to obtain the bounds. Each of the graded Betti numbers sums up all the multigraded ones of the same total degree $j$ for each homological degree $i$, i.e. $\beta_{i,j}(I)=\sum_{\deg(\mu)=j}\beta_{i,\mu}(I)$.

A good way to relate the information of the Hilbert series and Betti numbers is the use of generating functions. We first fix some notation and then express the bounds in terms of the Betti numbers $\beta_{i,j}$. First, consider a variable $x$ as a placeholder for the homological degree $i$. Thus  we define the graded Betti number generating function of an ideal $I$ as
$$G_I(x,t) = \sum_{i,j} \beta_{i,j}(S/I) x^i t^j\ \mbox{ for }\ i,j>0.$$
The numerator of the graded Hilbert series can now be expressed as
\begin{eqnarray}\label{eq:numerator}
\H_I(x,t) = -G_I(-x,t).
\end{eqnarray}
\begin{Notation}
In our setting, to simplify the notation $\H(x,t)$ denotes the numerator of the Hilbert series of the path ideal $\P_{L,q}$ and
$\tilde{\H}(x,t)$ denotes the numerator of the Hilbert series of the cut ideal $\C_{L,q}$. Similarly, we denote $G(x,t)$,
$\tilde{G}(x,t)$,  $\beta_{i,j}$ and $\tilde{\beta}_{i,j}$ for ${G}_{\C_{L,q}}(x,t)$, $G_{\P_{L,q}}(x,t)$, $\beta_{i,j}(S/\P_{L,q})$ and $\beta_{i,j}(S/\C_{L,q})$.
\end{Notation}
\begin{Remark}\label{rem:perc}
Using the path ideal, the path probability (percolation) is  given by
\begin{eqnarray}\label{eq:percolation}
\mathcal{P}(p)  =  \H(1,p).
\end{eqnarray}
Observe that this expression comes from the fact that the numerator of the Hilbert series represents the full inclusion inclusion-exclusion of the orthants with ``corner''
at the minimal paths and then replacing $t$ by $p$ translates this into the reliability function: see \cite{SW09}. 
\end{Remark}

On the other hand, as mentioned, the cut probability is given by
\begin{eqnarray}\label{eq:percolation2}
1-\mathcal{P}(p) = 1-\H(1,p).
\end{eqnarray}
Put briefly, the full Hilbert series gives the full operating set.
We also have the cut ideal and use tilde for the cut quantities. So we have for the cut ideal:
$$\tilde{G}(x,t) = \sum_{i,j} \tilde{\beta}_{i,j }x^i t^j\ \mbox{ for }\ {i,j}>0,$$
$$\tilde{\H}(x,t) =   -\tilde{G}(-x,t).$$
The probability of system failure is
$$\tilde{\mathcal{P}}(q) = \tilde{\H}(1,q).$$
So combining the formula we have two ways of expressing the probability $\mathcal{P}(p)$:
\begin{equation}\label{eq:dual}
\mathcal{P}(p)  = 1- \tilde{\mathcal{P}}(1-p).
\end{equation}
This is a manifestation of the Alexander duality in terms of probability.

To complete the notation, consider any power series in $x$, say $f(x) = \sum_{i \geq 0}c_i x^i$ and let
$T_m(f(x)) = \sum_{i=0}^m c_i x^i$, be the truncated version at $i=m$. Thus, powers of the ``dummy" variable $x$ can be used to pick out
the depth at which we truncate to get bounds:
$$  T_{2r+2}(\H(x,p))\vline_{\; x=1}  \leq \mathcal{P}(p) \leq T_{2r+1}(\H(x,p))\vline_{\; x=1}, r = 0,1, \ldots $$
$$  T_{2r+2}(\tilde{\H}(x,q))\vline_{\; x=1}  \leq \tilde{\mathcal{P}}(q) \leq T_{2r+1}(\tilde{\H}(x,q))\vline_{\; x=1}, r = 0,1, \ldots .$$
As we will see in detail in the \S\ref{sec:bounds}, the path bounds are accurate for small $p$, and the cut bounds for small $q$, or $p=1-q$ close to $1$.

\section{Tree percolation}\label{sec:tree}
Let us apply these techniques to a prominent example, namely percolation in complete $k$-ary trees.
A complete $k$-ary tree $T_{k,n}$ of depth $n$, is a tree with $n$ levels in which each node (except the leaves) has exactly $k$ children. Each edge between nodes is called a bond. 
See Figure~\ref{fig:k2} for $k=2$ and $n=3$. We are interested in  standard tree bond percolation on  $T_{k,n}$. Each bond has an independent  probability $p$ of being operative. A {\em percolation} is a path of bonds from the first generation (root) to the last (a leaf).  If we consider the unique minimal ways to connect each of the leaves with the root as minimal connecting events, then we want to find the probability of the union of events that contain at least one minimal connecting path from the root to a leaf. We will use an algebraic approach to solve this problem.

\begin{figure}[h!]  \begin{center}

\begin{tikzpicture}  [scale = .22, very thick = 15mm]

 \node (n0) at (-25,11) [Cgray] {};
\node (m1) at (-29,9)  [Cwhite] {$x_1 $};
\node (m2) at (-21,9)  [Cwhite] {$x_2$};
\node (n1) at (-31,6)  [Cgray] {};
\node (n2) at (-19,6)  [Cgray] {};
\node (m3) at (-33.5,4)  [Cwhite] {$x_3$};
\node (m4) at (-28.5,4)  [Cwhite] {$x_4$};
\node (n3) at (-34,1)  [Cgray] {};
\node (n4) at (-28,1)  [Cgray] {};
\node (m3) at (-21.5,4)  [Cwhite] {$x_5$};
\node (m4) at (-16.5,4)  [Cwhite] {$x_6$};
\node (n5) at (-22,1)  [Cgray] {};
\node (n6) at (-16,1)  [Cgray] {};
\node (n7) at (-35.8,-4)  [Cgray] {};
\node (n8) at (-32.2,-4)  [Cgray] {};
\node (m7) at (-35.8,-1)  [Cwhite] {$x_7$};
\node (m8) at (-32.2,-1)  [Cwhite] {$x_8$};
\node (m9) at (-29.8,-1)  [Cwhite] {$x_9$};
\node (m10) at (-26.1,-1)  [Cwhite] {$x_{10}$};
\node (n9) at (-29.8,-4)  [Cgray] {};
\node (n10) at (-26.2,-4)  [Cgray] {};
\node (n11) at (-23.8,-4)  [Cgray] {};
\node (n12) at (-20.2,-4)  [Cgray] {};
\node (n13) at (-17.8,-4)  [Cgray] {};
\node (n14) at (-14.2,-4)  [Cgray] {};
\node (m7) at (-36,-1)  [Cwhite] {$x_7$};
\node (m8) at (-32,-1)  [Cwhite] {$x_8$};
\node (m11) at (-23.6,-1)  [Cwhite] {$x_{11}$};
\node (m12) at (-20.2,-1)  [Cwhite] {$x_{12}$};
\node (m13) at (-17.7,-1)  [Cwhite] {$x_{13}$};
\node (m14) at (-14,-1)  [Cwhite] {$x_{14}$};

 \foreach \from/\to in {n0/n1,n0/n2,n1/n3,n1/n4,n2/n5,n2/n6, n3/n7,n3/n8,n4/n9,n4/n10, n5/n11,n5/n12,n6/n13,n6/n14}
 \draw[][->] (\from) -- (\to);
\end{tikzpicture} \end{center}

\caption{$T_{2,3}$}
\label{fig:k2}
\end{figure}

Much of the theory of percolation is about critical values. For the bond percolation on $T_{k,n}$, there is a critical value, denoted $p_c$,
such that for $0 \leq p \leq p_c$, as $n \rightarrow \infty$, the probability of a percolation tends to zero, whereas
for $p_c \leq p \leq 1$ the probability tends to a positive probability. This critical value is known to be $p_c = \frac{1}{k}$.
This is a classical result and is often covered in the theory of branching processes, where the positive probability is referred to as the probability of survival of a branching process, in which in every generation each individual has $k$ offsprings. For the general theory
of percolation see \cite{G99} and for work on  percolation on trees see \cite{lyons1990} and Chapter 5 of \cite{lyons2005probability}.

For the study of critical values in our algebraic setting, we use recurrence relationships for the Hilbert series giving  $\H_{k,n}(x,t)$ in terms of $\H_{k,n-1}(x,t)$. In the final section we use these recurrences to study the behaviour of the bounds, as $n \rightarrow \infty$, including a notion of asymptotic Betti numbers.

\subsection{The path ideal of $T_{k,n}$}\label{subsec:path}
Let us consider $T_{k,n}$ as a rooted graph with the edge set $E$ (the edges are oriented  going away from the root) as  in Figure~\ref{fig:k2}. We label each node with increasing integers, starting from the root, which has label $0$ and within the same level from left to right.
Each edge will be labelled with $x_i$, where
$i$ is the head of the edge, i.e., the edge is directed toward $i$.

\medskip

Let $K$ be a field and let $S=K[\xb]$ be the polynomial ring in the $m=|E|=\sum_{j=1}^n k^j$ variables $\{x_e: e \in E\}$.
The path ideal $I_{k,n}$ is the monomial ideal in $R$ generated by the monomials $x_{i_1}\cdots x_{i_n}$ where $0,i_1,\ldots,i_n$ is a unique path from the root to the leaf $i_n$. The ideal $I_{k,n}$ has then $k^n$ minimal generators (one for each leaf).

\begin{Notation}
For fixed integers $k$ and $n$, we fix $0$ as a source, and the set of leaves as targets instead of $q$ and $L$ from \S\ref{sec:preliminary}. The ideals $I_{k,n}$ and $J_{k,n}$ denote the corresponding path ideal $\P_{L,q}$, and the corresponding cut ideal $\C_{L,q}$ from \S\ref{sec:preliminary}.
\end{Notation}

\begin{Remark}{\rm
Let $R = K[\yb]$ be the polynomial ring over a field $K$ on $|V(T_{k,n})|$ variables. The path ideal of length $n$ associated to $T_{k,n}$ is the monomial ideal $I_{n+1} \subset S$ generated by monomials $y_{i_0}y_{i_1}\cdots y_{i_n}$ where ${i_0},{i_1},\ldots, {i_n}$ is a path in $T_{k,n}$. Such ideals are studied in \cite{Tina,Tuyl}. Note that if $n$ is the depth of the tree, then $y_{i_0}$ is the variable corresponding to the root.  In particular $I_{n+1}$ is isomorphic to our ideal $I_{k,n}$ 
 under the induced isomorphism
\[
\varphi \colon \mathcal{G}(I_{n+1}) \rightarrow \mathcal{G}(I_{k,n})\ { \rm with }\ y_{i_0}y_{i_1}\cdots y_{i_n}\mapsto x_{i_1}\cdots x_{i_n}
\]
where $\mathcal{G}(I)$ denotes the generating set of $I$, and $x_{i_\ell}$ is the variable corresponding to the edge between the vertices $y_{i_{\ell-1}}$ and $y_{i_\ell}$.
}
\end{Remark}

\begin{Lemma}
Let $T_{k,n}$ be a $k$-ary tree and $I_{k,n}$ its path ideal. Then we have that
\begin{itemize}
\item[(i)] the Taylor resolution of $I_{k,n}$ is minimal for all $k$ and $n$,

\item[(ii)] the Betti numbers are given by  $\beta_i(I_{k,n}) ={k^n \choose i}$,

\item[(iii)] The length of the resolution, i.e., the projective dimension of $I_{k,n}$ is ${k^n}$.
\end{itemize}
\end{Lemma}

\begin{proof} Each minimal generator $m_{\alpha}$ of $I_{k,n}$ has a variable that appears only in it, namely the one corresponding to the edge $ij$ where $j$ is the leaf in $m_\alpha$. Hence the monomials corresponding to the least common multiple of any two different sets of minimal generators are different, hence the multidegrees of the generators of the Taylor resolution of $I_{k,n}$ are all different, and hence the Taylor resolution of $I_{k,n}$ is minimal. The Betti numbers of $I_{k,n}$ are then the ranks of the modules in the Taylor resolution.
\end{proof}

\subsection{Path formulae}\label{sec:path}
Here we read the resolution of $I_{k,n}$ as a tensor product of the resolutions of ideals arising from $I_{k,n-1}$ to obtain the generating function and a recursive formula for the ideal's Betti numbers (see Appendix~\ref{sec:tensor} for some basic facts about tensor products of chain complexes).
\begin{Theorem}\label{prop:proof}
The total Betti numbers of $I_{k,n}$ are $\beta_i(I_{k,n}) ={k^n \choose i}$ and the graded Betti numbers $\beta_{i,j}$ can be determined recursively as:
\[
\beta_{i-1,j}(I_{k,n})=\beta_{i,j}(S/I_{k,n})=\sum_{s=1}^k\sum_{(i_1,\ldots,i_s)\in A_{i,s}\atop (j_1,\cdots,j_s)\in B_{j,s}}
{k \choose s}
\beta_{i_1,j_1}(S/I_{k,n-1})\cdots \beta_{i_s,j_s}(S/I_{k,n-1})\ ,
\]
where
\[
A_{i,s}=\{(i_1,\ldots,i_s):\ i_1+\cdots+i_s=i,\ i_1,\ldots,i_s>0\}\ {\rm and}
\]
\[
\ B_{j,s}=\{(j_1,\ldots,j_s):\ j_1+\cdots+j_s=j-s\}.
\]
\end{Theorem}

\begin{proof} Assume that $x_{1},x_{2},\ldots,x_{k}$ are the variables corresponding to the edges of tree connected to the root. Then
\[
I_{k,n}=x_1I_{k,n-1}^{(1)}+x_2I_{k,n-1}^{(2)}+\cdots+x_kI_{k,n-1}^{(k)},
\]
where each $I_{k,n-1}^{(i)}$ is a tree ideal associated to a $k$-ary tree of depth $n-1$.
Note that their corresponding trees are disjoint, so the ideals $I_{k,n-1}^{(i)}$ live in disjoint polynomial rings. Therefore, the resolution of $I_{k,n}$ is the tensor product of the resolutions of the ideals $x_iI_{k,n-1}^{(i)}$.
On the other hand, $\beta_{i,j}(x_iI_{k,n-1}^{(i)})=\beta_{i,j+1}(I_{k,n-1})$ for all $k$ and $n$. Thus  the statement is an immediate consequence of Lemma \ref{initial}(iii).\end{proof}


\begin{Remark}\label{rem:gen}
Let us denote by $G_{k,n}=\sum_{i,j}\beta_{ij}(I_{k,n})x^it^j$  the generating function for the Betti numbers of the ideal $I_{k,n}$. We also denote the numerator of the graded Hilbert series of $I_{k,n}$ by $\H_{k,n}$. Note that by Remark~\ref{eq:numerator} we have 
$$\H_{k,n}(x,t) = -G_{k,n}(-x,t).$$
\end{Remark}
We recall that the ideal $I_{k,1}=\langle x_1,x_2, \ldots,x_k\rangle$ is generated by $k$ variables. Thus $\beta_{i,j}={k\choose i}$ if $i=j$, and it is zero otherwise. Therefore
\[
G_{k,1}(x,t)=(1+tx)^k-1.
\]
From Theorem~\ref{prop:proof} and the above argument we obtain the following compact result.
\begin{Theorem} \label{thm:gen}
The generating function for the Betti numbers of $I_{k,n}$ for all $k$ and $n$, is equal to
\begin{equation}\label{eq:gen-path}
G_{k,n}(x,t) = (1 + tG_{k,n-1}(x,t))^k-1.
\end{equation}
\end{Theorem}

\subsection{Cut ideal and cut formulae}\label{sec:cut}

 As explained in Proposition~\ref{prop:primdecom} the cut ideal is the Alexander dual ideal of  the path ideal $I_{k,n}$. We consider the following problem: Given a probability $p_i$ for each edge $i$ in $T_{k,n}$ to be operative, we want to find the probability of {\em disconnecting} the root with all leaves of graph. If we consider all minimal possible ways to disconnect the leaves with the root as minimal connecting events, then what we want to find is the probability of the union of events that does not contain any  path connecting the root to a leaf.  As before, we consider $p_i=p$ for all $i$. 

Here we read the ideal $J_{k,n}$ as the Alexander dual of the tree ideal $I_{k,n}$ studied in \S\ref{sec:path} to obtain the generating function and a recursive formula for its Hilbert series.

For all $k$ and $n$, the generating function for the Betti numbers of the ideal $J_{k,n}$ is denoted by $\tilde{G}_{k,n}$.
We recall that the ideal $I_{k,1}=\langle x_1,x_2, \ldots,x_k\rangle$ is generated by $k$ variables, and its dual is $J_{k,1}=\langle x_1x_2\cdots x_k\rangle$. Thus $\beta_{0,0}(S/J_{k,1})=1$, $\beta_{1,k}(S/J_{k,1})=1$, and it is zero otherwise. Therefore
$\tilde{G}_{k,1}(x,t)=t^kx$.

\begin{Theorem} The generating function of the Betti numbers of $J_{k,n}$  for all $k$ and all $n>1$, is equal to
\begin{equation}\label{eq:gen-cut}
\tilde{G}_{k,n}(x,t) = x^{-(k-1)}\big((1+tx)(1 + \tilde{G}_{k,n-1}(x,t))-1\big)^k.
\end{equation}
\end{Theorem}

\begin{proof}
Assume that $x_{1},x_{2},\ldots,x_{k}$ are the variables corresponding to the edges of tree connected to the root, and $I_{k,n-1}^{(i)}$ is a tree ideal associated to a $k$-ary tree of depth $n-1$. We denote $J_{k,n-1}^{(i)}$ for the Alexander dual of the ideal $I_{k,n-1}^{(i)}$. Thus the Alexander dual of the ideal $x_iI_{k,n-1}^{(i)}$ is equal to $\langle x_i\rangle+J_{k,n-1}^{(i)}$, because $x_i$ doesn't appear in the support of any monomial from the generating set of $I_{k,n-1}^{(i)}$. Thus the numerator of the graded Hilbert series (\ref{hilb}) of  $\langle x_i\rangle+J_{k,n-1}^{(i)}$ is equal to
\[
(1-tx)(1+\tilde{G}_{k,n-1}(-x,t)).
\]
On the other hand, $J_{k,n}$ can be written as the multiplication of the ideals $\langle x_i\rangle+J_{k,n-1}^{(i)}$ living in the polynomial rings on disjoint variables:
\[
J_{k,n}=\big(\langle x_1\rangle+J_{k,n-1}^{(1)}\big) \cdots \big(\langle x_k\rangle+J_{k,n-1}^{(k)}\big).
\]
Now applying Lemma~\ref{initial2}(iii), then the same argument as Theorem~\ref{thm:gen} implies that the minimal free resolution of the ideal $J_{k,n}$ is the tensor product of that of $x_i+J_{k,n-1}^{(i)}$ and so we have
\[
\HS_{J_{k,n}}(x,t)=\frac{1+(-x)^{-(k-1)} \big((1-tx)(1+\tilde{G}_{k,n-1}(-x,t))-1\big)^k}{(1-t)^d}. \quad
\]
Therefore
\begin{eqnarray}\label{hilbcut}
\H_{J_{k,n}}(x,t)=-(-x)^{-(k-1)} \big((1-tx)(1+\tilde{G}_{k,n-1}(-x,t))-1\big)^k
\end{eqnarray}
and
\[
\tilde{G}(x,n)=x^{-(k-1)} \big((1+tx)(1+\tilde{G}_{k,n-1}(x,t))-1\big)^k.\qed
\]

\end{proof}

\section{Bounds and  critical values}\label{sec:bounds}
Using the same notation as in Remark~\ref{rem:gen}  and applying Remark~\ref{rem:perc}
the percolation probability of the path ideal $I_{k,n}$ is then given by
$$\mathcal{P}_{k,n}(p) = \H_{k,n}(1,p).$$
Similarly we have
$$\tilde{\mathcal{P}}_{k,n}(q) = \tilde{\H}_{k,n}(1,q).$$
From \eqref{eq:dual} and Theorem~\ref{thm:gen} we have the iterative formula
\begin{equation}
\H_{k,n}(x,t) =1- \big(1- t(\H_{k,n-1}(x,t))\big)^k.
\end{equation}
This gives the formula for $\mathcal{P}_{k,n}(p)$:
\begin{equation}
\mathcal{P}_{k,n}(p) = 1 - \left(1-p\ \mathcal{P}_{k,n-1}(p)\right)^k,
\end{equation}
which is well known from the theory of branching process \cite{harris2002}. We will have in mind the classical asymptotic form for $\mathcal{P}_{k,n}(p)$. As $n \rightarrow \infty$, and  for fixed $k$, $\mathcal{P}_{k,n}(p)$ converges to the function:
$$\mathcal{P}_{k,\infty}(p) = \max(0,1-u),$$
where $u$ is the solution of
$$u=(1-p(1-u))^k.$$
The value $p_c = \frac{1}{k}$ is the maximum value of $p$ for which $\mathcal{P}_{k,\infty}(p) = 0$.

For the bounds given in  \S\ref{sec:preliminary} we write the path and cut bounds, respectively as
$$B_{k,n,m}(p) = T_m (\H(x,p)) \vline_{\;x=1},$$
$$C_{k,n,m}(q) = T_m (\tilde{\H}(x,q)) \vline_{\;x=1}.$$

We now discuss how the bounds for percolation based on $B_{k,n,m}$ and $ C_{k,n,m}$
behave. As a brief guide to a quite technical section the following is an informal list of the main features found by the authors
\begin{enumerate}
\item The path bounds $B_{k,n,m}$ are accurate as $p \rightarrow 0$.
\item The cut bound,  $C_{k,n,m}$, are accurate as $p \rightarrow 1$ ($q=1-p \rightarrow 0$).
\item The path bounds display critical behaviour at the critical value $p_c = \frac{1}{k}$ in that they diverge away from the true probability, as $n$ increases and can only be controlled by taking higher depth $m$ .
\item The cut bounds reveal a new type of critical value $p_k^* = 1- q^*_k> p_c$.
\end{enumerate}

All these results are  consequences of having the iterative formulae (\ref{eq:gen-path}) and (\ref{eq:gen-cut}).  It should also be noted that the path bounds are easier to handle than the cut bounds, which is a consequence of the Taylor resolution (standard inclusion-exclusion) being the minimal free resolution in the path case, which is not true in the cut case. By working on the first few bounds,  we can obtain exact formula and limits in some cases.

\begin{Example}
Figure~\ref{fig:bounds} gives an example combining the path and the cut bounds for $k=2$, $n=4$ and depth at $m= 3,4$. Observe that together with the curve showing the true probability of percolation there are four curves plotted together in this figure, two on the left side of the figure i.e. probability $p$ closer to $0$ and two on the right side i.e. probability $p$ closer to $1$. To cope with the divergence near the critical value the upper and lower bounds are truncated respectively at 1 and 0.  The upper bounds are in green: $m=3$ for the path bound on the left and $m=4$ for the cut bound on the right. The lower bounds are in blue: $m=4$ for the path bound on the left, and $m=3$ for the cut bound on the right. The central red curve is the true probability of percolation.
 
\end{Example}

We begin with some formulae for the path case. Multiplying the bounds by a truncated version of the product
$(1-kp)(1-k^2p^2)(1-k^3p^3) \cdots$
leads to tractable formulae. Interestingly, the inverse of this infinite product is the generating function for integer partitions.
For the path bounds we have the following, for $n=1,2,3$:
\begin{eqnarray*}
B_{k,n,1}(p)& = & p^n k^n \\
B_{k,n,2}(p)(1-kp) & =  & p^n k^n  \left( 1- \frac{1}{2}(3k-1)p +\frac{1}{2} (k-1)k^n p^{n+1} \right)               \\
B_{k,n,3}(p)(1-kp)(1-k^2p^2) & = & p^n k^n (1- \frac{1}{2}(3k-1)p -\frac{1}{6}(k+1)(5k-2)p^2\\
& & +\frac{1}{6}k (11k^2-6k+1)p^3+ \frac{1}{2} k^n(k-1)p^{n+1}   \\
& & -\frac{1}{2} k^n(k-1)^2 p^{n+2}    -\frac{1}{2} k^{n+1}(2k-1)(k-1)p^{n+3} \\
& & + \frac{1}{6}k^{2n}(2k-1)(k-1)p^{2n+2} \\
& & + \frac{1}{6}k^{2n+1}(k-1)(k-2)p^{2n+3} )\ .
  \end{eqnarray*}
The general formula, whose proof is omitted, is
$$B_{k,n,m} \prod_{i=1}^{m-1}(1-k^ip^i)  = p^nk^n\left(Q_{k,p,m}(k,p) + O(p^{n+1})\right),$$
where $Q_{k,p,m}(k,p)$ is a polynomial in $p$, the degree of which depends only on $k$ and $m$. This gives some asymptotics
as $n \rightarrow \infty$. To aid this we set $p = \frac{R}{k}$, having in mind that $\frac{1}{k}$ is the critical value.

After a little algebra we have the following formulae
\begin{eqnarray*}
B_{k,n,1}(p) & = & R^n \\
B_{k,n,2}(p) & = & R^n \; \left( 1-     \frac{1}{2}\;\frac{k-1}{k} \; \frac{R}{1-R} \right) + \mbox{O}(R^{2n+1}) \\
B_{k,n,3}(p) & = & R^n \;\left (1- \frac{1}{6} \frac{k-1}{k^2}\; \frac{R(-5R^2k+R^2-Rk+2R+3k)}{(1-R)(1-R^2)}  \right) + \mbox{O}(R^{2n+1})
\end{eqnarray*}

For fixed $R < 1$ the first bound  $B_{k,n,1} \rightarrow 0$, as expected. It is instructive to let $k \rightarrow \infty$, again while keeping $R<1$ fixed. Combining the first two bounds ($m=2,3$) we have asymptotically, we obtain
$$ R^n \left(  1- \frac{1}{2} \frac{R}{(1-R)} \right) \leq B_{\infty, n, \infty} \leq R^n  \left(  1- \frac{1}{6}\frac{R(3-R -5R^2)}{(1-R)(1-R^2)}  \right).$$
The bounds agree to order $R^{n+1}$, the left bound reaches zero at $R=\frac{1}{2}$ and the bounds diverge to $\pm \infty$ as $R \rightarrow 1$. It has a pole at $R=1$,
 for all the bounds except the first from which we claim that $p=p_c = \frac{1}{k}$ is a critical value for the path bounds, albeit buried under a basic  $R^n$ convergence rate for $R<1$.

\medskip

Now let us consider the cut bounds. We start with a basic form of the iteration for the cut generating function  from \eqref{hilbcut}:
\begin{eqnarray}
g(x,t,k,u) = (-1)^k\frac{1}{x^{k-1}} \left( (1-tx)(1-u) -1 \right)^k.
\end{eqnarray}
Then  the successive values of $\H_{k,n}(x,t)$ are given by the recurrence relation:
$$\H_{k,n}(x,t) = g(x,t,k,\H_{k,n-1}(x,t)).$$
Note that
$$\H_{k,1}(x,t) = t^k x.$$
The main difficulty is that although we have recurrence for $\H_{k,n+1}(x,y)$ there is not in general such a nice
formula for the truncated version which is given by extracting the Taylor expansion in $x$ up to degree $m$.

However, there is one simple case,
namely the first upper cut bound, i.e. $m=2$, which we denote by $C_{k,n,2}(q)$ . We have for all $n \geq1$
\begin{eqnarray} \label{C_k,n,2}
C_{k,n,2}(q) = \left(C_{k,n-1,2}(q)+q \right)^k.
\end{eqnarray}
At any fixed $k,q$ the value $C_{k,n,2}(q)$ increases with $n$.  There is a critical  value $q=q^*_{k,n,2}$. For any $q$ above this value $C_{k,n,2}(q) \rightarrow \infty$ as $n \rightarrow \infty$. In the interval $0 \leq q \leq  q^*_{k,n,2}$, $C_{k,n,m}$  tends from below to the solution to
$$z = (z+q)^k.$$
We can solve this explicitly for $q$, giving $q
 = z^{\frac{1}{k}} -1$. The critical values of $q$ and $z$ are found,
by solving $\frac{dz}{dq} = 0$ and we obtain:
$$q^*_k = \frac{k-1}{k^2} \;k^{\frac{k-2}{k-1}},$$
which plays a key role in the cut theory.

\begin{figure}[tbp]
\centering
 {\includegraphics[width=90mm]{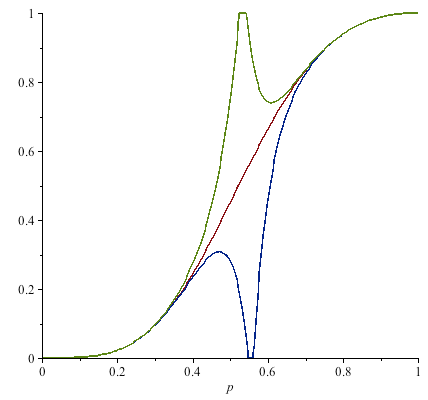}}
\caption{Upper and lower path and cut bounds for $k=2$, $n=4$, and $m=3,4$. }\label{fig:bounds}
\end{figure}

\begin{Example}
For $n=2,3$ and $k=2$ we obtain:
$$
\begin{array}{rcl}
\H_{2,2}(x,t) & = & t^2(t+1)^2x-2t^4(t+1)x^2+t^6x^3 \\
\H_{2,3}(x,t) & =  & t^2(t^3+2t^2+t+1)^2x-2t^4(t^3+2t^2+t+1)(t+1)(3t+1)x^2+ \\
 &  & t^6(15t^4+40t^3+36t^2+18t+5)x^3-2t^8(10t^3+20t^2+12t+3)x^4 + \\
& & t^{10}(15t^2+20t+6)x^5-2t^{12}(3t+2)x^6+t^{14}x^7
\end{array}
$$
\end{Example}
\begin{Remark}
Note that for all $i$ and $j$, we can read off the Betti numbers $\beta_{i,j}(S/J_{k,n})$ as the coefficients
of $x^i t^j$ in the polynomial $\H_{k,n}(x,t)$.
\end{Remark}

The limiting behaviour of  $C_{k,n,m}(q)$ as $n\rightarrow \infty$ will be considered below. 

\begin{Lemma}
For fixed $q$ and $k=m=2$, the lower asymptotic bound of $C_{k,n,m}(q)$ as $n\rightarrow \infty$ is
$$C_{2,\infty,2}(q) = \frac{1}{2} - q^2+\frac{1}{2} \frac{6q^2+2q-1}{\sqrt{1-4q}}.$$
\end{Lemma}
\begin{proof}
We consider the solution in $u$, of $u-g(x,t,k,u)=0$. 
As $k=2$, we use the solution
\begin{equation}\
u^*(x,t) = \frac{1}{2}\frac{(2t^2x-2t+1+\sqrt{-4t+1+4t^2x})x}{(1-xt)^2}\ .
\end{equation}
If we expand in powers of $x$ we obtain:
$$
\left(\frac{1}{2}\sqrt{1-4t}\right)x -\left(\frac{t}{2}\sqrt{1-4t} +t^2 -t \sqrt{1-4t}\right)x^2
 + \mbox{O}(x^3).
$$
The cut bounds are obtained by truncating such expansions and setting $t=q$ and $x=1$. Thus the above expansion gives the $k=2, m=2$ lower asymptotic bound:
\[C_{2,\infty,2}(q) = \frac{1}{2} - q^2+\frac{1}{2} \frac{6q^2+2q-1}{\sqrt{1-4q}}.\qed\]

\end{proof}
Further expansion in $t$ gives asymptotic Betti numbers which we will cover in the next section. We collect these informal results into the following.
\begin{Theorem}\label{prop:lim}
For fixed $k,m$, and  $q$, the cut bounds $C_{k,n,m}(q)$ converge to the function which is the truncated Taylor expansion
of the smallest solution, in $u$, of the equation $u-g(x,q,k,u)=0$, provided $0\leq q \leq q^*_k = \frac{k-1}{k^2} \;k^{\frac{k-2}{k-1}}$.
For $q^*_k < q \leq 1$, $C_{k,n,m}(q) \rightarrow \infty$ when $m$ is odd and $-\infty$ when $m$ is even.
\end{Theorem}

\begin{proof}
The case $m=1$  above is a good guide, because it gives the first term in the Taylor expansions and further analysis
shows that it also gives pole governing the expansion for any  $m > 1$.  Although we cannot get closed forms for the solutions of 
$u-g(x,q,k,u)=0$ for $k >3$, nonetheless we can show the presence of a pole at $q^*_{k}$.

 We use the shorthand  $g(u) = g(x,t,k,u)$. Then, $g(u)$ is convex and increasing in $u$ in the region $0 \leq x,t \leq 1$ and $g(0) = xt^k$. Moreover $u_1 =xt^k$, where $u_1$ is the starting value in the iteration $u_{n+1} = g(u_n)$. Suppose first that, under that suitable conditions on $x,t$, the equation $u=g(u)$, has at least one solution and let $u^*$, be the smallest (there can be no more than two, by the convexity of $g(u)$). Then, for the dynamic system $u_{n+1} = g(u_n)$ the iterate $u_n$ converge upwards to $u^*$. The complication is that the existence of the solution $x^*$ depends on $x$ and $q$. We know that $x^*$ does not exist if and only if
$g(u)>u$, for all $u >0$.

We consider the case $k=2$, for which $q_2^* = \frac{1}{4}$. The function
$g(u)- u$
has roots
$$\frac{1}{2}\frac{(2t^2x-2t+1+\sqrt{4t^2x-4t+1})x}{1+t^2x^2-2tx}, - \frac{1}{2}\frac{(-2t^2x+2t-1+\sqrt{4t^2x-4t+1})x}{1+t^2x^2-2tx}\ .$$
In the region $0 \leq q \leq \frac{1}{4}$ these roots exist for all $x > 0$, except when $t = \frac{1}{x}$, which can be eliminated by taking
$x$ sufficiently small. When $q >\frac{1}{4}$, however, the roots are complex for $x$ sufficiently small, noting that the smallest root is
$$\left(-t+\frac{1}{2}-\frac{1}{2}\sqrt{1-4t}\right)x + O(x^2).$$
In that case $g(u) > u$ and  $u_n$ diverges to infinity. The case of general $k$ proceeds along the same lines.
The fact that the bounds, which are achieved at $x=1$, converge in the ``good" region, $0 \leq q \leq q^*_k $ follows by standard analysis on the uniform convergence of power series. For  $q^*_k < q \leq 1$, the divergence of $C_{k,n,m}(q)$, follows immediately from the divergence of $u_n$, with sign dependent on $m$
\end{proof}
Figure~\ref{fig:critical_bound} shows an example of the behaviour of the cut bounds for $k=2$ and relatively modest value $n=6$. The cut upper  bound for $m=3$ (depicted in green),  and the cut  lower bound for $m=4$ (depicted in red) are already approaching the vertical line at the critical value $q=\frac{1}{4}$.
\begin{figure}[tbp]
\centering
 {\includegraphics[width=80mm]{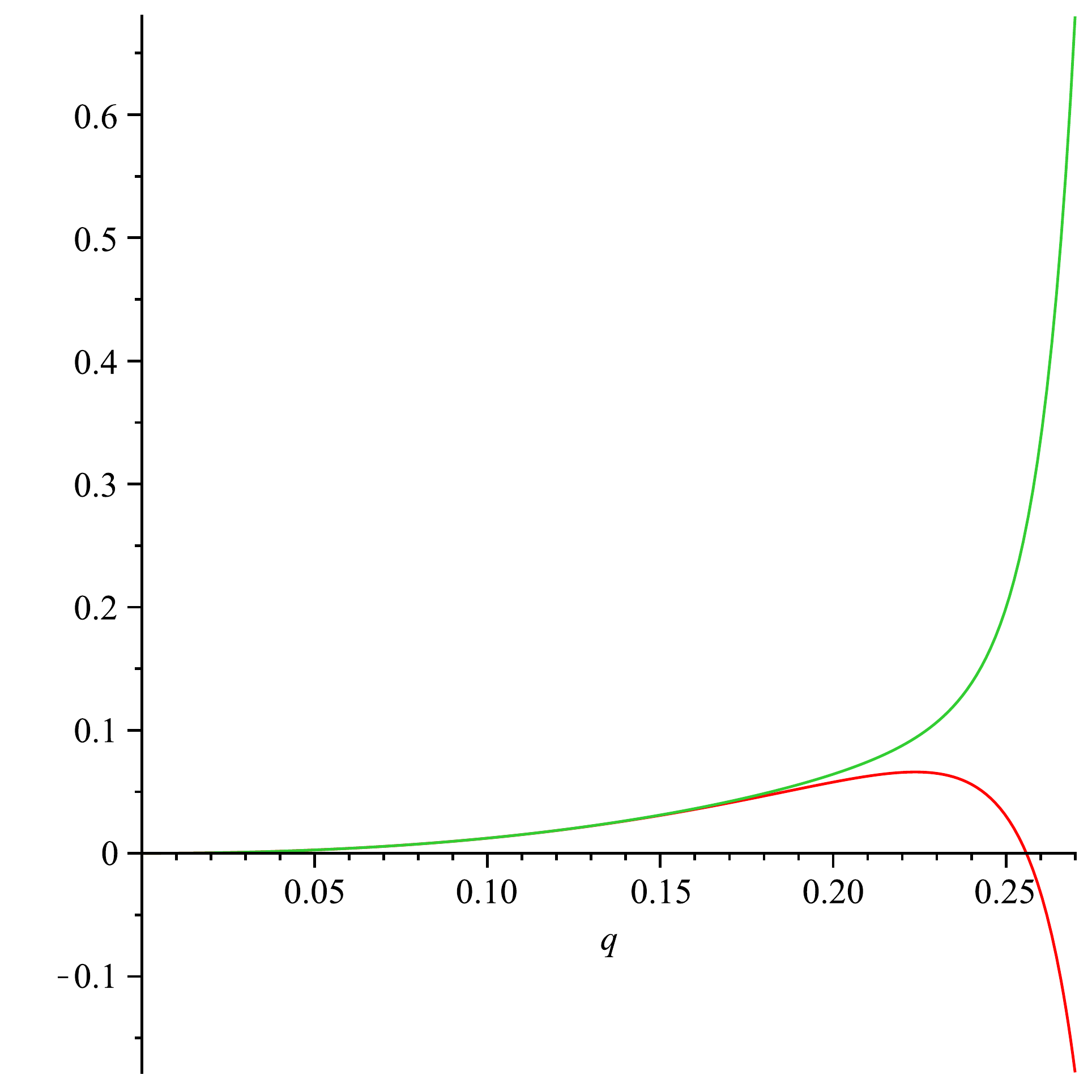}}
\caption{Approaching the $q^*=\frac{1}{4}$ critical cut bounds for $n=6, \;k=2, \;$ $m=3$ (cut upper bound in green) and $m=4$ (cut lower bound in red).}\label{fig:critical_bound}
\end{figure}
\subsection{Asymptotic Betti number: cut case.}
For the case $k=2$ we use the appropriate generating function in (\ref{eq:gen-cut}) with the discussion
in the last section to give the graded Betti number generating function:
\begin{equation}\label{Q2infty}
G_{2,\infty}(x,t) = - \frac{1}{2}\; \frac{(2t-1+2t^2x+(
\sqrt{1-4t-4t^2x})x}{(1+xt)^2}.
\end{equation}
There are more complex formulae for $k>2$.

This generating function enables us to derive a combinatorial formula for the coefficients, which we shall call the {\em asymptotic graded Betti numbers}.

\begin{Theorem}\label{thm:cutBetti}
The asymptotic graded Betti numbers for the cut ideal on a binary $(k=2)$ tree are given by
$$\beta_{i,j}(J_{2,\infty})= \frac{[2(j-i)]!}{(j-i+1)(j-i)(j-i)! (j-2i)!(i-1)!}, \mbox{ for } j \geq 2i \mbox{ and
zero otherwise.}$$
\end{Theorem}
\begin{proof}
We temporarily make the transformation $t = \frac{z}{1+y}, x = \frac{y(1+y)}{z}$ in the generating function (\ref{Q2infty})
giving:
\begin{equation}
G_{2,\infty}(x,t) = \frac{1-\sqrt{1-4z}}{2z}\frac{y}{1+y} - \frac{y}{1+y}\ .
\end{equation}
We recognize the first term on the right hand side as the generating function of the Catalan numbers $c_r = \frac{1}{r+1} { 2r \choose r} z^r$,
$$ \frac{1-\sqrt{1-4z}}{2z} = \sum_{r=0}^{\infty}c_r z^r.$$
We now back substitute $y=xt, z=t(1+xt)$ so that the generating function for the $c_r$ becomes: $\sum_{r=0}^{\infty} c_r t^r (1+xt)^r$. Expanding each term $(1+xt)^r$ into binomial terms and incorporating the other terms $y$ we find the Betti numbers $\beta_{i,j} (J_{2,\infty})$ as the coefficient of $x^i x^j$, for $j \geq 2i$.
\begin{eqnarray*}
\beta_{i,j}(J_{2, \infty})& = & c_{j-i} \sum_{r=0}^{i-1}(-1)^{i+j-1}{j-i\choose r}\\
 & = & c_{j-i} \frac{(j-i-1)!}{(i-1)!(j-2i)!},
\end{eqnarray*}
for $j\geq 2i$ and zero otherwise. Using the form of $c_{j-i}$, we obtain the result.
\end{proof}

\vspace{-.2cm}
\begin{Theorem}\label{asymp}
In the region $0\leq q \leq q^*_k$, for each $i$ there is a maximal integer $N(n,i)$ such that for any $j < N(n,j)$
$$\beta_{i,j}(J_{k,n}) = \beta_{i,j}(J_{k,\infty}).$$
Furthermore $N(n,j)$ is increasing in $n$, for fixed $i$.
\end{Theorem}

\begin{proof}
This follows from the uniform convergence of the power series derived form $u_n$ and the fact that the coefficients
are integers.
\end{proof}

\vspace{-.1cm}
The following tables show the graded Betti numbers $\beta_{ij}$ of the cut ideal for $k=2$,  $n = 2, \ldots, 4 $ and ranges of
values of $i=1, \ldots, 7$ and $j= 2, \ldots, 14$. Note that its $(i,j)$-entry is simply $\beta_{i,i+j}(S/J_{k,n})$.  The table for $n=2,3$ are complete. The last table gives the asymptotic Betti numbers.

\begin{center}
\begin{tabular}{|l||l|l|l|l|l|}
  \hline
  $i \backslash j$ & 0 & 1 & 2 & 3  \\
  \hline\hline
  ${\rm total}$ & 1 & 4 & 4 & 1 \\
  \hline
  \hline
  0 & 1 &    &  &    \\
  1 &    & 1 &  &  \\
  2 &    & 2 & 2 &    \\
  3 &    & 1 & 2 & 1 \\
  \hline
\end{tabular}
\end{center}
\begin{center}
$n=2$
\end{center}

\begin{center}
\begin{tabular}{|r||r|r|r|r|r|r|r|r|r|r|r|r|r|}
  \hline
  $i \backslash j$ & 0 & 1 & 2 & 3 & 4 & 5 & 6 & 7  \\
    \hline\hline
  ${\rm total}$ & 1 & 25 & 80 & 114  & 90 & 41 & 10 & 1 \\
\hline
  \hline
  0 &1&    &      &  &  &  &  &   \\
  1 &  &  1&      &  &  &  &  &   \\
  2 &  &  2& 2   &  &  &  &  &   \\
  3 &  &  5& 10 & 5 &  &  &  &   \\
  4 &  &  6& 18 & 18 & 6 &  &  &   \\
  5 &  &  6& 24 & 36 & 24 & 6 &  &  \\
  6 &  &  4& 20 & 40 & 40 & 20 & 4 &   \\
  7 &  &  1& 6  & 15 & 20 & 15 & 6 & 1  \\
  \hline
\end{tabular}
\end{center}
\begin{center}
$n=3$
\end{center}

\begin{center}
\begin{tabular}{|l||l|l|l|l|l|l|l|l|l|l|l|l|l|l|}
  \hline
$i \backslash j$ & 0 & 1 & 2 & 3 & 4 & 5 & 6 & 7 \\
  \hline\hline
  ${\rm total}$ & 1 & 676 & 5460 & 21113  & 51348 & 87288 & 109314 & 103726   \\ \hline
 \hline
  0 &1&    &      &  &  &  &   &  \\
  1 &  &  1&      &  &  &  &   &  \\
  2 &  &  2& 2   &  &  &  &   &  \\
  3 &  &  5& 10 & 5 &  &  & &  \\
  4 &  &14& 42 & 42 & 14   &  &  & \\
  5 &  &26&104& 156 & 104 & 26 &  & \\
  6 &  &44&220& 440 & 440 & 220 & 44   & \\
  7 &  &69&414& 1035 & 1380 & 1035 & 414 & 69 \\
  8 &  &94&658& 1974 & 3290 & 3290 &  1974  & 658  \\
  9 &  &   114 & 912 & 3192 & 6384 & 7980 &  6384 & 3192 \\
  10& &116 &1044&  4176&  9744& 14616 & 14616  & 9744 \\
  11 &&94 & 940&  4230& 11280& 19740&  23688 & 19740 \\
  12 && 60 &660&  3300&  9900& 19800&  27720 & 27720  \\
  13 &&28 & 336&  1848&  6160& 13860&  22176 &25872 \\
  14 &&8 & 104&   624&  2288 & 5720&  10296 & 13728 \\
    \hline
\end{tabular}
\end{center}
\begin{center}
$n=4$
\end{center}
\vspace{3mm}

\begin{center}
\begin{tabular}{|l||l|l|l|l|l|l|l|l|l|l|l|l|l|l|}
  \hline
$i \backslash j$ & 0 & 1 & 2 & 3 & 4 & 5 & 6 & 7 \\
  \hline\hline
  ${\rm total}$ & 1 & 458329 & 8308144 & 73630338  & 424216050 & 1783078865 & 5818552406 & 15319701281
\\ \hline
 \hline
  0 &1&       &      &  &  &  &   &  \\
  1 &  &   1  &      &  &  &  &   &  \\
  2 &  &  2& 2   &  &  &  &   &  \\
  3 &  &  5& 10 & 5 &  &  & &  \\
  4 &  &14& 42 & 42 & 14   &  & &  \\
  5 &  &42&168 & 252 & 168 & 42 & &  \\
  6 &  &100&500& 1000 & 1000 & 500 & 100  &   \\
  7 &  &221&1326& 3315 & 4420 & 3315 & 1326 & 221  \\
  8 &  &470&3290& 9870 & 16450 & 16450 &9870 & 3290   \\
  9 &  &   958 & 7664 &26824 & 53648 & 67060 &  53648 & 26824 \\
  10& &1860 &16740&  66960&  156240& 234360 & 234360 & 156240 \\
  11 &&3434 & 34340&  154530& 412080& 721140& 865368  & 721140 \\
  12 && 6036 &66396&  331980&  995940& 1991880&  2788632  & 2788632 \\
  13 &&10068 & 120816&  664488&  2214960& 4983660& 7973856 & 9302832  \\
  14 &&15864 & 206232&  1237392&  4537104 & 11342760&  20416968 & 27222624  \\
    \hline
\end{tabular}
\end{center}
\begin{center}
$n=5$
\end{center}
\vspace{3mm}

\begin{center}
\begin{tabular}{|l||l|l|l|l|l|l|l|l|l|l|l|l|l|l|}
  \hline
$i \backslash j$ & 0 & 1 & 2 & 3 & 4 & 5 & 6 & 7 \\
 \hline
  0 &1&       &      &  &  &  &   &  \\
  1 &  &   1  &      &  &  &  &   &  \\
  2 &  &  2& 2   &  &  &  &   &  \\
  3 &  &  5& 10 & 5 &  &  & &  \\
  4 &  &14& 42 & 42 & 14   &  & &  \\
  5 &  &42&168 & 252 & 168 & 42 & &  \\
  6 &  &132&660& 1320& 1320 & 660 & 132  &   \\
  7 &  &429&2574& 6435 & 8580 & 6435 & 2574 & 429  \\
  8 &  &1430&10010& 30030 & 50050 & 50050 &30030 & 10010   \\
  9 &  & 4862 & 38896 &136136 & 272272 & 340340 &  272272 & 136136 \\
  10& &16796&151164&  604656&  1410864& 2116296 & 2116296 & 1410864 \\
  11 &&58786& 587860&  2645370& 7054320& 12345060& 14814072  & 12345060 \\
  12 && 208012&2288132&  11440660&  34321980& 68643960&  96101544  & 96101544 \\
  13 &&742900& 8914800&  49031400&  163438000& 367735500& 588376800 & 686439600  \\
  14 &&2674440&34767720&  208606320&  764889840 & 1912224600&  3442004280 & 4589339040  \\
    \hline
\end{tabular}
\end{center}
\begin{center}
$n=\infty$
\end{center}
\vspace{3mm}

Let us make a final remark on the relation of the Betti numbers of $J_{2,n}$ and $J_{2,\infty}$ with Mandelbrot and Catalan numbers. This will make evident the interplay of algebra, combinatorics and asymptotics that permeates this paper.

Mandelbrot numbers are defined as follows: The Mandelbrot Set is a fractal formed by iterating the following polynomials:
\begin{eqnarray*}
z_0(q)&=&0\\
z_n(q)&=&z_{n-1}(q)^2+q
\end{eqnarray*}
The first few polynomials in the sequence are 
\begin{eqnarray*}
z_0(q)&=&0\\
z_1(q)&=&q\\
z_2(q)&=&q^2+q\\
z_3(q)&=&(q^2+q)^2+q=q+q^2+2q^3+q^4\\
z_4(q)&=&(q+q^2+2q^3+q^4)^2+q=q+q^2+2q^3+5q^4+6q^5+6q^6+4q^7+q^8\\
\end{eqnarray*}

We see that the coefficient of $q^j$ in $z_i(q)$ is the $(i,j)$-Mandelbrot number, and we denote it by $M_{i,j}$.
\begin{Lemma}
$C_{1,n,2}(q)=z_{n+1}(q)-q$.
\end{Lemma}

\begin{proof}
The proof goes by induction. For $n=1$ we have that $C_{1,1,2}(q)=q^2$ and $z_2(q)=q^2+q$. 
Now, $C_{1,n+1,2}(q)=(C_{1,n+1,2}(q)+q)^2$ and by the induction step, this is $(z_{n+1}(q)-q+q)^2=z_{n+1}(q)^2$. And on the other hand, we have that, by definition, $z_{n+2}(q)=z_{n+1}(q)^2+q$.
\end{proof}

This identification of polynomials gives us an expression of the Betti numbers of $J_{2,n}$ in terms of Mandelbrot numbers:

\begin{Corollary}
$\beta_{i,i+j}(J_{2,n})=M_{n+1,j+1}{{j-1}\choose{i-1}}$.
\end{Corollary}

\begin{proof}
We have that $\beta_{1,j+1}(J_{2,n})$ is the coefficient of $q^j$ in $C_{1,n,2}(q)$. Since $C_{1,n,2}(q)=z_{n+1}(q)-q$, the result holds for $i=1$. For $i>1$ we only need to multiply by the binomial coefficient, which comes from the recursive expression of the Betti numbers of $J_{2,n}$.
\end{proof}

In the asymptotic case $n=\infty$ we have a similar expression, where Mandelbrot numbers are substituted by Catalan numbers:

\begin{Corollary}
$\beta_{i,i+j}(J_{2,\infty})=c_{j}{{j-1}\choose{i-1}}$.
\end{Corollary}

\begin{proof}
Just substitute $j$ by $i+j$ in the last expression of the proof of Theorem~\ref{thm:cutBetti}. 
\end{proof}

The last two corollaries together with the asymptotic study of the Betti numbers imply that:

\begin{Corollary}
$lim_{n\rightarrow\infty} M_{n,j}=c_j$.
\end{Corollary}

\appendix
\section{Tensor product of complexes}\label{sec:tensor}
 To keep the paper self-contained, we review here some basic and relevant notions of tensor product of resolutions.
We begin by recalling the tensor product of resolutions  from \cite{Eisenbud}.  The tensor product of two chain complexes $(A,d_1)$ and $(B,d_2)$, say $A\otimes B$, is formed by taking all products
$A_i \otimes B_j$ and letting $(A \otimes B)_k = \bigoplus_{i+j=k} A_i\otimes B_j$. The  differential maps are  defined as $\partial(a\otimes b) = d_1a \otimes b + (-1)^i a \otimes d_2b$, when $a\in A_i$. Then we have $\partial^2 = 0$ and $\partial$ induces a natural map $\partial:H(A)\otimes H(B) \rightarrow H(A\otimes B)$ such that $\partial(a\otimes b)=a\otimes b$. If $a=d_1c$ is a boundary and $b$ is a cycle, then $a \otimes b = \partial(c\otimes b)$ is again a  boundary which shows that
$\partial$ is well-defined.

\begin{Lemma}\label{initial}
Let $I_i\subseteq S_i$ be a monomial ideal in the polynomial ring $S_i$ for $i=1,\ldots,r$, and let $I=I_1+\cdots+I_r$ be the ideal in the polynomial ring $S= S_1\otimes\cdots\otimes S_r$. Assume that $\mathcal{F}_{i}$ is the minimal free resolution of $S/I_i$ for all $i$.
Then the minimal free resolution of $S/I$ is obtained by $\mathcal{F}_{1}\otimes \mathcal{F}_{2}\otimes \cdots\otimes \mathcal{F}_{r}$. In particular,
\begin{itemize}
\item[$(i)$]
$
\beta_{i,j}(S/I)=\sum_{(i_1,\ldots,i_r)\in A_{i,r}\atop (j_1,\cdots,j_r)\in B_{j,r}}
\beta_{i_1,j_1}(S/I_{1})\cdots \beta_{i_r,j_r}(S/I_{r}),
$
where
\[
A_{i,r}=\{(i_1,\ldots,i_r):\ i_1+\cdots+i_r=i\}\ {\rm and}\ B_{j,r}=\{(j_1,\ldots,j_r):\ j_1+\cdots+j_r=j\}.
\]

\item[$(ii)$] The   Hilbert
series of $S/I$ is
$
\HS_{I}(t)=\frac{\prod_{i=1}^r (1+Q_{S/I_i}(t))}{(1-t)^d},
$
where $\HS_{I_i}(t)=\frac{1+Q_{S/I_i}(t)}{(1-t)^d}$ and $d$ is the number of variables of the ring $S$. Note that here we are looking at the ideals $I_i$ in the polynomial ring $S$.
\end{itemize}
In particular, if all the ideals are the same (but in polynomial rings on disjoint set of variables) then we have 
$
\HS_{I}(x,t)=\frac{(1+Q_{S/I_i}(t))^r}{(1-t)^d} and
$

\begin{itemize}
\item[$(iii)$]
$
\beta_{i,j}(S/I)=\sum_{s=1}^r\sum_{(i_1,\ldots,i_s)\in A_{i,s}\atop (j_1,\cdots,j_s)\in B_{j,s}}
{r \choose s}
\beta_{i_1,j_1}(S/I_{1})\cdots \beta_{i_s,j_s}(S/I_{1}),
$
where
\[
A_{i,s}=\{(i_1,\ldots,i_s):\ i_1+\cdots+i_s=i,\ i_1,\ldots,i_s>0\}\ {\rm and}
\]
\[
 B_{j,s}=\{(j_1,\ldots,j_s):\ j_1+\cdots+j_s=j\}.
\]
\end{itemize}
\end{Lemma}

\begin{proof} The proof is by induction on $r$. Assume that $r>1$. Since differential maps of the tensor complex are defined in terms of differential maps of $\mathcal{F}_\ell$'s, the minimality of the tensor complex follows by the minimality of the resolutions of all components.
On the other hand, these ideals live in rings with disjoint variables which implies that ${\rm Tor}_i(S/(I_1+\cdots+I_{r-1}),S/I_r)=0$ for $i>0$, and so the constructed complex is indeed a minimal free resolution for $S/I$.
\end{proof}

In Lemma~\ref{initial} if we replace  $I_1+\cdots+I_r$  and $S/I$ by $I_1I_2\cdots I_r$ and $I$, then an analogous statement holds.  The original statement appeared in Habilitationsschrift of J\"urgen Herzog in 1974, and the proof is similar to the proof of above lemma.
\begin{Lemma}\label{initial2}
Let $I_i\subseteq S_i$ be a monomial ideal in the polynomial ring $S_i$ for $i=1,\ldots,r$, and let $I=I_1 I_2\cdots I_r$ be the ideal in the polynomial ring $S= S_1\otimes\cdots\otimes S_r$. Assume that $\mathcal{F}_{i}$ is the minimal free resolution of $I_i$ for all $i$.
Then the minimal free resolution of $I$ is obtained by $\mathcal{F}_{1}\otimes \mathcal{F}_{2}\otimes \cdots\otimes \mathcal{F}_{r}$. In particular,
\begin{itemize}
\item[$(i)$]
$
\beta_{i,j}(I)=\sum_{(i_1,\ldots,i_r)\in A_{i,r}\atop (j_1,\cdots,j_r)\in B_{j,r}}
\beta_{i_1,j_1}(I_{1})\cdots \beta_{i_r,j_r}(I_{r}),
$
where
\[
A_{i,r}=\{(i_1,\ldots,i_r):\ i_1+\cdots+i_r=i\}\ {\rm and}
\]
\[
\ B_{j,r}=\{(j_1,\ldots,j_r):\ j_1+\cdots+j_r=j\}.
\]

\item[$(ii)$] The   Hilbert
series of $S/I$ is
$
\HS_{I}(x,t)=\frac{1+(-x)^{-(r-1)}\prod_{i=1}^r Q_{S/I_i}(t)}{(1-t)^d},
$
where $\HS_{I_i}(t)=\frac{1+Q_{S/I_i}(t)}{(1-t)^d}$ and $d$ is the number of variables of the ring $S$. Note that here we are looking at the ideals $I_i$ in the polynomial ring $S$.

\item[$(iii)$]
In particular, if all the ideals are the same (but in polynomial rings on disjoint set of variables) then we have

$
\HS_{I}(x,t)=\frac{1+(-x)^{-(r-1)} \big(Q_{S/I_i}(t)\big)^r}{(1-t)^d}.
$
\end{itemize}
\end{Lemma}

\begin{Example} Let  $I=\langle x_1x_2x_3, x_2x_4,x_1x_5x_6\rangle$ and $J=\langle y_1y_2y_3,y_2y_4y_5\rangle$. The Betti tables of $S/I$, $S/J$, $S/(I+J)$, and $S/IJ$ are as follows:

\vspace{3mm}
\begin{center}
\tiny{
\begin{tabular}{|l||l|l|l|l|l|}
  \hline
  $i \backslash j$ & 0 & 1 & 2 & 3  \\
  \hline\hline
  ${\rm total}$ & 1 & 3 & 3 & 1 \\
  \hline
  \hline
  0 & 1 &    &  &    \\
  1 &    & 1 &  &  \\
  2 &    & 2 & 1 &    \\
  3 &    &  & 2 & 1 \\
  \hline
\end{tabular}
\quad
\begin{tabular}{|l||l|l|l|l|l|}
  \hline
  $i \backslash j$ & 0 & 1 & 2  \\
  \hline\hline
  ${\rm total}$ & 1 & 2 &  1 \\
  \hline
  \hline
  0 & 1 &    &    \\
  1 &  & & \\
  2 &    & 2 &     \\
  3 &    &   & 1 \\
  \hline
\end{tabular}
\quad
\begin{tabular}{|l||l|l|l|l|l|l|l|}
  \hline
  $i \backslash j$ & 0 & 1 & 2 & 3 &4&5 \\
  \hline\hline
  ${\rm total}$ & 1 & 5 & 10 & 10 & 5 & 1\\
  \hline
  \hline
  0 & 1 &    &  &    & & \\
  1 &    & 1 &  &  & & \\
  2 &    & 4 & 1 &   & &  \\
  3 &    &  & 5 & 1 & & \\
  4 &    &  & 4 & 3 & & \\
  3 &    &  &  & 6 & 3 & \\
  3 &    &  &  &  & 2 & 1\\

  \hline
\end{tabular}
\quad
\begin{tabular}{|l||l|l|l|l|l|l|l|}
  \hline
  $i \backslash j$ & 0 & 1 & 2 & 3 & 4  \\
  \hline\hline
  ${\rm total}$ & 1 & 6 &  9 & 5 & 1 \\
  \hline
  \hline
  0 & 1 &    &    & & \\
  1 &  & & & & \\
  2 &  & & & & \\
  3 &    &  &   & &  \\
  4 &    & 2  &  & & \\
  5 &  & 4& 3& & \\
  6 &  & &6 &3 & \\
  7 &  & & & 2&1 \\

  \hline
\end{tabular}
}
\end{center}


Note that, the $(i,j)$-entry of the table corresponding to $S/I$ is simply $\beta_{i,i+j}(S/I)=\beta_{i-1,i+j}(I)$, and we have only listed the 
entries corresponding to the non-zero Betti numbers.
Since
\[ \HS_I(x,t)=\frac{1-x(t^2+2t^3)+x^2(t^4+2t^5)-x^3t^6}{(1-t)^{12}}\quad{\rm and}\quad
 \HS_J(x,t)=\frac{1-2xt^3+x^2t^5}{(1-t)^{12}}\ ,
\]
we have that
\begin{eqnarray*}
\HS_{I+J}(x,t)&=&\frac{(1-x(t^2+2t^3)+x^2(t^4+2t^5)-x^3t^6))(1-2xt^3+x^2t^5)}{(1-t)^{12}}\\&=&\frac{
1-x(t^{2}+4t^{3})+x^{2}(t^{4}+5t^{5}+4t^{6})-x^{3}(t^{6}+3t^{7}+6t^{8})+x^{4}(3 t^{9}+2 t^{10})-x^{5} t^{11}}{(1-t)^{12}}\ .
\end{eqnarray*}
Note that the above formula includes the graded Betti numbers. For example:
\begin{eqnarray*}
\beta_{1}(S/(I+J))&=&\beta_{1,3}(S/(I+J))+\beta_{1,2}(S/(I+J))\\
&=&\big(\beta_{0}(S/I)\beta_{1,3}(S/J)+\beta_{1,3}(S/I)\beta_{0}(S/J)\big)+\beta_{1,2}(S/I)\beta_{0}(S/J)
\\
&=&(2+2)+1=5,
\end{eqnarray*}
which is encoded as the coefficient of $-x$ (for $t=1$) in the above formula.


\smallskip

Similarly,
we have that
\begin{eqnarray*}
\HS_{IJ}(x,t)&=&\frac{1-x^{-1}(-x(t^2+2t^3)+x^2(t^4+2t^5)-x^3t^6)(-2xt^3+x^2t^5)}{(1-t)^{12}}\\&=&\frac{1-x(2t^5+4t^6)+x^2(3t^7+6t^8)-x^3(3t^9+2t^{10})+x^4t^{11}}{(1-t)^{12}}\ .
\end{eqnarray*}

From the above formula we obtain the graded Betti numbers. For example:
\begin{eqnarray*}
\beta_{0}(IJ)&=&\beta_{0,5}(IJ)+\beta_{0,6}(IJ)=\beta_{0,2}(I)\beta_{0,3}(J)+\beta_{0,3}(I)\beta_{0,3}(J)
= 2+4=6,
\end{eqnarray*}
which is encoded as the coefficient of $-x$ in the above formula.
\begin{eqnarray*}
\beta_{1}(IJ)&=&\beta_{1,7}(IJ)+\beta_{1,8}(IJ)\\&=&\big(\beta_{0,2}(I)\beta_{1,5}(J)+\beta_{1,4}(I)\beta_{0,3}(J)\big)+
\big(\beta_{1,5}(I)\beta_{0,3}(J)+\beta_{0,3}(I)\beta_{1,5}(J)\big)
\\&=& \big(1+2)+(4+2)=9.
\end{eqnarray*}
The term  $x^2(3t^7+6t^8)$ in the above formula shows that $\beta_{1,7}=3$ and $\beta_{1,8}=6$.

\end{Example}

\bibliographystyle{alpha}
\bibliography{allocation_2016}
\end{document}